\documentclass{article}
\usepackage{graphicx}
\usepackage{cite}
\IfFileExists{authblk.sty}{\usepackage{authblk}}{}
\usepackage{amsthm}
\usepackage{amsmath}
\usepackage{amssymb}
\usepackage[left=3cm, right=3cm, top=1.0in, bottom=1.0in]{geometry}
\usepackage[colorlinks=true]{hyperref}

\date{}
\title{Strict Convexity for Solution of Liouville-Type Dirichlet Problems}
\author{Jiahuan Li, Shuning Xu}

\newcommand{\keywords}[1]{\par\quad\textbf{Keywords:} #1}

\newcommand{\R}{\mathbb R}
\newcommand{\C}{\mathbb C}
\newcommand{\tr}{\operatorname{tr}}
\newcommand{\rank}{\operatorname{rank}}
\newcommand{\arcosh}{\operatorname{arcosh}}
\newcommand{\sech}{\operatorname{sech}}

\begin{document}
\maketitle

\newtheorem{theorem}{Theorem}[section]
\newtheorem{definition}[theorem]{Definition}
\newtheorem{lemma}[theorem]{Lemma}
\newtheorem{corollary}[theorem]{Corollary}
\newtheorem{example}[theorem]{Example}
\newtheorem{proposition}[theorem]{Proposition}
\newtheorem{conjecture}[theorem]{Conjecture}
\newtheorem{remark}[theorem]{Remark}

\begin{abstract}
We identify a common convexity structure for three exponential Dirichlet
problems on smooth uniformly strictly convex domains: the Liouville equation
$\Delta u=e^u$, the real equation $\sigma_2(D^2u)=e^{2u}$, and its complex
counterpart $\sigma_2(u_{i\bar j})=e^{2u}$. In each case $u<0$ in the domain
and $u=0$ on the boundary. We prove that
\[
w=-\operatorname{arcosh}(e^{-u/2})
\]
is strictly convex in the underlying real variables. The argument combines
domain deformation, constant-rank theory, inverse-convexity estimates, radial
ball models, boundary strict convexity, and local $C^2$ stability. 
\end{abstract}


\section{Introduction}

Convexity properties of solutions to elliptic equations connect nonlinear
analysis with convex geometry. A classical example is the theorem of
Brascamp and Lieb \cite{BrascampLieb1976}, which yields log-concavity of the
first Dirichlet eigenfunction of the Laplacian on a convex domain. Such
concavity properties are closely related to Brunn--Minkowski inequalities for
eigenvalues and other variational quantities; see
\cite{Colesanti2005,Jerison1996,Salani2005,Schneider1993}.

The study of convexity properties of solutions of partial differential
equations predates the systematic power-concavity theory of the 1980s.
In 1971, Makar-Limanov proved that the square root of the torsion
function is concave on planar convex domains \cite{ML71}. In 1976,
Brascamp and Lieb showed that log-concavity is preserved by the heat
flow and, in particular, obtained the log-concavity of the first
Dirichlet eigenfunction on convex domains
\cite{BrascampLieb1976}. In the 1980s, Korevaar developed the convexity
maximum principle, while Kennington formulated a systematic theory of
power concavity for semilinear Dirichlet problems
\cite{Korevaar1983,Ken85}. In particular, Kennington recorded a result
of Grant Keady showing that if
\[
\Delta u=e^u\quad\hbox{in }\Omega,
\qquad
u=k\quad\hbox{on }\partial\Omega,
\]
then $\sqrt{k-u}$ is concave \cite[Theorem~5.3]{Ken85}. Thus, when
$k=0$, this result already gives the concavity of $\sqrt{-u}$ for the
Liouville problem considered below. The present arcosh transform yields
a stronger strict-convexity statement and, consequently, strengthens
this classical square-root concavity to strict concavity.

For strict convexity, a foundational step was the work of Caffarelli
and Friedman, who combined a deformation argument with a
two-dimensional constant-rank theorem for convex solutions of
semilinear elliptic equations \cite{CaffarelliFriedman1985}. Their work
is the direct methodological precursor of the
deformation--constant-rank strategy used in the present paper.
Korevaar and Lewis subsequently extended the constant-rank theorem to
higher dimensions \cite{KorevaarLewis1987}.At the macroscopic level, Alvarez, Lasry, and Lions introduced an
inverse-convexity condition involving the map
$A\mapsto F(A^{-1})$ and used it to establish convexity of viscosity solutions under state-constraint boundary conditions
\cite{AlvarezLasryLions1997}. This inverse-matrix structure  later became an important ingredient in the fully nonlinear constant-rank theory. In the fully nonlinear setting, Guan and Ma developed the constant-rank method for Hessian
equations in their study of the Christoffel--Minkowski problem
\cite{GuanMa2003}, and Caffarelli, Guan, and Ma established a general
constant-rank theorem for fully nonlinear elliptic equations
\cite{CaffarelliGuanMa2007}. Bian and Guan later formulated the
microscopic convexity principle under a general structural condition
\cite{BianGuan2009,BianGuan2010}.

The first equation considered here is the Liouville-type problem
\begin{equation}\label{eq:laplace-main}
\Delta u=e^u,
\qquad u<0\quad\hbox{in }\Omega,
\qquad u=0\quad\hbox{on }\partial\Omega.
\end{equation}
In dimension two, this equation has a direct conformal interpretation. If
$g=e^u g_{\mathrm{Eucl}}$, then its Gaussian curvature is
\[
K_g=-\frac12 e^{-u}\Delta u;
\]
hence \eqref{eq:laplace-main} prescribes the constant curvature
$K_g=-1/2$. The boundary condition fixes the scale of the conformal factor.
Our result shows that this negative-curvature equation carries a hidden
convex geometry after an equation-adapted change of variables.

We next consider Hessian operators. For $1\leq k\leq n$ and
$u\in C^2$, the $k$-Hessian operator is
\[
\sigma_k(D^2u)=\sum_{1\leq i_1<\cdots<i_k\leq n}
\lambda_{i_1}(D^2u)\cdots\lambda_{i_k}(D^2u).
\]
For $1<k\leq n$, the equation is fully nonlinear and elliptic only on the
appropriate G{\aa}rding cone. The Dirichlet theory was developed by
Caffarelli, Nirenberg, and Spruck \cite{CaffarelliNirenbergSpruck1985}, while
Wang studied the corresponding Hessian eigenvalue problem \cite{Wang1994}.
Convexity and Brunn--Minkowski results for Hessian equations were obtained in
\cite{MaXu2008,LiuMaXu2010,Salani2012}; higher-dimensional real and complex
$\sigma_2$ developments appear in \cite{limasa,chenlima}.

Complex Hessian equations have their own analytic subtleties, since
admissibility is imposed on the Hermitian Hessian while our conclusion is a
statement about the full real Hessian. Relevant developments in a priori
estimates, Dirichlet problems, Liouville theorems, and eigenvalue problems may
be found in
\cite{HouMaWu2010,DinewKolodziej2011,DinewKolodziej2012,CollinsPicard2022,BadianeZeriahi2023,ChuLiuMcCleerey2024}.

The original solution of a Hessian equation is generally not the convex
quantity. For example, the natural transforms for
\[
\sigma_2(D^2u)=\lambda(-u)^2
\quad\hbox{and}\quad
\sigma_2(D^2u)=1
\]
are, respectively, $-\log(-u)$ and $-(-u)^{1/2}$. These examples belong to
the power-concavity framework initiated by Kennington. The exponential
source is nonhomogeneous in that scale, however, and the power transform does
not remove its explicit dependence on the solution. The appropriate change
of variables is
\begin{equation}\label{eq:main-transform}
w=-\arcosh(e^{-u/2}),
\qquad
u=h(w):=-2\log(\cosh w),
\qquad w<0.
\end{equation}
It is characterized by the identity
\begin{equation}\label{eq:h-ode}
h''(w)=-2e^{h(w)},
\end{equation}
which absorbs the exponential source into the transformed equation. With
\[
s(w)=-\sinh w\cosh w>0,
\]
problem \eqref{eq:laplace-main} becomes
\begin{equation}\label{eq:laplace-transformed-intro}
s(w)\Delta w-|Dw|^2-\frac12=0.
\end{equation}
For the real exponential $\sigma_2$ equation
\begin{equation}\label{eq:real-main}
\sigma_2(D^2u)=e^{2u},
\qquad u<0\quad\hbox{in }\Omega,
\qquad u=0\quad\hbox{on }\partial\Omega,
\end{equation}
the transformed equation is
\begin{equation}\label{eq:real-transformed-intro}
s(w)^2\sigma_2(D^2w)
-s(w)\tr(P_{\R}(Dw)D^2w)-\frac14=0,
\qquad P_{\R}(p)=|p|^2I-p\otimes p.
\end{equation}
The same scalar transform applies to the complex equation
\begin{equation}\label{eq:complex-main}
\sigma_2(u_{i\bar j})=e^{2u},
\qquad u<0\quad\hbox{in }\Omega,
\qquad u=0\quad\hbox{on }\partial\Omega.
\end{equation}
The complex Hessian must then be viewed as a compression of the real Hessian;
this is why the complex result below concerns strict convexity in the
underlying real variables.

Our main results are as follows.

\begin{theorem}\label{thm:laplace-main}
Let $\Omega\subset\mathbb R^N$, $N\geq 2$, be a bounded, smooth,
uniformly strictly convex domain. Let
$u\in C^2(\Omega)\cap C^0(\overline{\Omega})$
be a solution of~\eqref{eq:laplace-main}.
Then 
$u\in C^\infty(\overline{\Omega}),$
and$-\operatorname{arcosh}(e^{-u/2})$is strictly convex in $\Omega$.
\end{theorem}

\begin{theorem}\label{thm:real-main}
Let $\Omega\subset\mathbb R^n$, $n\geq 3$, be a bounded, smooth,
uniformly strictly convex domain. Let
$u\in C^2(\Omega)\cap C^0(\overline{\Omega})$
be a $\sigma_2$-admissible solution of \eqref{eq:real-main}.
Then 
$u\in C^\infty(\overline{\Omega}),$
and
$-\operatorname{arcosh}(e^{-u/2})$
is strictly convex in $\Omega$.
\end{theorem}
\begin{theorem}\label{thm:complex-main}
Let $\Omega\subset\mathbb C^m$, $m\geq 2$, be a bounded smooth
domain that is uniformly strictly convex when regarded as a domain
in $\mathbb R^{2m}$. Let
$u\in C^2(\Omega)\cap C^0(\overline{\Omega})$
be a complex $\sigma_2$-admissible solution of\eqref{eq:complex-main}.
Then 
$u\in C^\infty(\overline{\Omega}),$
and
$w=-\operatorname{arcosh}(e^{-u/2})$
is strictly convex in $\Omega$ with respect to the underlying real
variables. Equivalently,
\[
D_{\mathbb R}^2w>0
\qquad\hbox{in }\Omega.
\]
\end{theorem}

The proof has three components. First, the transformed equations satisfy the
structural condition in the Bian--Guan microscopic convexity principle
\cite{BianGuan2009,BianGuan2010}. For the real and complex Hessian equations,
this step uses inverse-convexity statements arising from G{\aa}rding's theory
of hyperbolic polynomials and its convex-analytic refinements
\cite{Garding1959,BauschkeGulerLewisSendov2001,Renegar2006}. Second, direct
radial ODE arguments give strict convexity on balls, while a boundary
calculation gives strict convexity near the boundary of every uniformly
strictly convex domain. Third, local $C^2$ stability and Minkowski deformation
reduce the general case to the ball; constant rank prevents loss of strict
convexity at the closedness step.

As a scalar consequence of the main theorems, the arcosh transform also
recovers the classical square-root concavity. More precisely, strict
convexity of
$-\operatorname{arcosh}(e^{-u/2})$
implies strict concavity of $\sqrt{-u}$. Moreover, the exponent
$1/2$ is optimal for the Liouville problem. These consequences are
established in Subsection~\ref{subsec:square-root-consequence}.

The paper is organized as follows. Section~2 collects the algebraic
and constant-rank tools and establishes the square-root concavity
consequence and its sharpness. Section~3 derives the transformed
equations and verifies the corresponding constant-rank structures
for the Liouville equation and the real and complex
$\sigma_2$-Hessian equations. Section~4 establishes boundary strict
convexity and analyzes the radial ball models. Section~5 develops
the domain-deformation framework, proves the required local
$C^2$ stability results, and completes the proofs of the main
theorems.

\section*{Acknowledgments}
The authors thank Professor Xi-Nan Ma for bringing this question to their attention. They were supported by the National Natural Science Foundation of China [grant number 2025YFA1017601]. 
\section{Preliminaries}

\subsection{The real and Hermitian \texorpdfstring{$\sigma_2$}{sigma2} operators}

For a real symmetric matrix $M\in S^n$, let $\lambda(M)=(\lambda_1,\ldots,\lambda_n)$ be its eigenvalues and set
\[
\sigma_1(M)=\sum_i\lambda_i,\qquad
\sigma_2(M)=\sum_{i<j}\lambda_i\lambda_j.
\]
Equivalently,
\begin{equation}\label{eq:sigma-two-formula}
\sigma_2(M)=\frac12\{(\tr M)^2-\tr(M^2)\}.
\end{equation}
The G{\aa}rding cone is
\[
\Gamma_2=\{\lambda\in\R^n:\sigma_1(\lambda)>0,\ \sigma_2(\lambda)>0\}.
\]
A $C^2$ function $u$ on a real domain is called $\sigma_2$-admissible if $\lambda(D^2u)\in\Gamma_2$ pointwise.

For a Hermitian matrix $M\in\mathcal H^n$, $\sigma_1(M)$ and $\sigma_2(M)$ are defined using the real eigenvalues of $M$. A real-valued function $u$ on a complex domain is complex $\sigma_2$-admissible if $\lambda(u_{i\bar j})\in\Gamma_2$ at every point.

The first Newton tensor is
\begin{equation}\label{eq:newton-tensor}
T_1(M)=\sigma_1(M)I-M.
\end{equation}
On $\Gamma_2$ one has
\begin{equation}\label{eq:newton-positive}
T_1(M)>0.
\end{equation}
Indeed, after diagonalizing $M$, the eigenvalues of $T_1(M)$ are
\[
\sum_{j\neq i}\lambda_j(M),\qquad i=1,\ldots,n,
\]
which are positive on $\Gamma_2$. We shall repeatedly use the monotonicity of $\sigma_2$ on the admissible branch: if $M\in\Gamma_2$ and $N\geq0$, then
\[
\sigma_2(M+N)\geq\sigma_2(M).
\]
This follows by integrating
\[
\frac{d}{dt}\sigma_2(M+tN)=\tr(T_1(M+tN)N)\geq0,
\]
because $M+tN\in\Gamma_2$ and $T_1(M+tN)>0$ for $0\leq t\leq1$.

\subsection{Rank-one formulas}

For $p\in\R^n$ define
\begin{equation}\label{eq:real-P}
P_{\R}(p)=|p|^2I-p\otimes p.
\end{equation}
For every real symmetric matrix $r$ and every $t\in\R$,
\begin{equation}\label{eq:real-rank-one}
\sigma_2(r+t\,p\otimes p)
=
\sigma_2(r)+t\tr(P_{\R}(p)r).
\end{equation}
Indeed, using \eqref{eq:sigma-two-formula}, the coefficient of $t$ is
\[
|p|^2\tr r-p^Trp=\tr((|p|^2I-p\otimes p)r),
\]
and the coefficient of $t^2$ is $\sigma_2(p\otimes p)=0$.

For $q\in\C^m$ define the Hermitian matrix
\begin{equation}\label{eq:complex-P}
P_{\C}(q)=|q|^2I-q\otimes\bar q.
\end{equation}
For every Hermitian matrix $M$,
\begin{equation}\label{eq:complex-rank-one}
\sigma_2(M-tq\otimes\bar q)
=
\sigma_2(M)-t\tr(P_{\C}(q)M).
\end{equation}
This follows from
\[
\frac{d}{dt}\sigma_2(M-tq\otimes\bar q)
=-\tr(T_1(M)q\otimes\bar q)
=-\bar q^TT_1(M)q,
\]
and
\[
\bar q^TT_1(M)q=|q|^2\tr M-\bar q^TMq
=\tr(P_{\C}(q)M),
\]
while the quadratic term vanishes because $q\otimes\bar q$ has rank one.

\subsection{Complex Hessians as compressed real Hessians}

Identify $\C^m$ with $\R^{2m}$ by $z_j=x_j+iy_j$. If
\[
H=\begin{pmatrix}U&V\\ V^T&W\end{pmatrix}\in S^{2m},
\]
define
\begin{equation}\label{eq:compression-block}
C(H)=\frac14\{U+W+i(V-V^T)\}.
\end{equation}
Equivalently,
\begin{equation}\label{eq:compression-R}
C(H)=\frac14R^*HR,\qquad
R=\begin{pmatrix}I\\ iI\end{pmatrix}\in\C^{2m\times m}.
\end{equation}
If $H=D^2_{\R}u$, then
\begin{equation}\label{eq:compression-hessian}
C(D^2_{\R}u)=(u_{i\bar j})_{1\leq i,j\leq m}.
\end{equation}
If $\lambda=(a,b)\in\R^m\times\R^m$ and $\zeta=a-ib\in\C^m$, then
\begin{equation}\label{eq:compression-rank-one}
C(\lambda\otimes\lambda)=\frac14\zeta\otimes\bar\zeta.
\end{equation}
Thus a nonzero real rank-one positive semidefinite direction is compressed to a nonzero Hermitian rank-one positive semidefinite direction.

\subsection{Bian--Guan constant rank theorem}

We use the following application form of the Bian--Guan microscopic convexity principle.

\begin{theorem}\label{thm:bian-guan}
Let $D\subset\R^N$ be connected and let $v\in C^{3,1}(D)$ be a convex solution of
\[
F(D^2v,Dv,v,x)=0,
\]
where $F$ is $C^{2,1}$ in its variables. Assume:
\begin{itemize}
\item[(i)] $F$ is elliptic along the solution, namely $(F^{\alpha\beta})>0$;
\item[(ii)] $F(0,Dv,v,x)\neq0$ along the solution;
\item[(iii)] for each fixed relevant gradient $p$, the set
\[
\{(A,z,x)\in S^N_{++}\times\R\times D:
F(A^{-1},p,z,x)\leq0\}
\]
is locally convex near the relevant points.
\end{itemize}
Then $\rank D^2v$ is constant in $D$.
\end{theorem}

This is the form used in the real Hessian and complex Hessian deformation arguments; see \cite{BianGuan2009,BianGuan2010}.

\subsection{Inverse-convexity inputs}

We first record the real inverse-convexity statement. Let $\alpha\in\R^n$, $|\alpha|=1$, and
\[
Q_\alpha=I-\alpha\otimes\alpha.
\]

\begin{proposition}[\cite{limasa}]\label{prop:real-inverse}
Let $n\geq3$ and $\mu>0$. The function
\begin{equation}\label{eq:real-inverse-function}
A\mapsto
\frac{\sigma_2(A^{-1})-\mu}{\tr(Q_\alpha A^{-1})}
\end{equation}
is convex on $S^n_{++}$.
\end{proposition}

This formulation follows from G{\aa}rding quotient concavity together with
the reciprocal-trace principle of Alvarez, Lasry, and Lions; see also
\cite{AlvarezLasryLions1997,BauschkeGulerLewisSendov2001,Garding1959,Renegar2006}.

The complex input is the compressed inverse-convexity theorem.

\begin{proposition}[\cite{chenlima}]\label{prop:complex-inverse}
Let $m\geq2$. For every $q\in\C^m$ and every $\mu>0$, the set
\begin{equation}\label{eq:complex-K}
K_{q,\mu}
=
\{A\in S^{2m}_{++}:
\sigma_2(C(A^{-1})-q\otimes\bar q)\leq\mu\}
\end{equation}
is convex.
\end{proposition}
\subsection{Square-Root Consequence and Sharpness}
\label{subsec:square-root-consequence}

We first show that strict convexity of the arcosh transform implies
the classical square-root concavity.

\begin{remark}
\label{rem:square-root-consequence}
Let $u<0$, and set
\[
W=-\operatorname{arcosh}(e^{-u/2}),
\qquad
V=-\sqrt{-u}.
\]
If $W$ is strictly convex, then $V$ is also strictly convex.
Consequently, each of Theorems~1.1--1.3 implies that
$\sqrt{-u}$ is strictly concave.

Indeed, since
\[
e^{-u/2}=\cosh W,
\]
we have
\[
-u=2\log(\cosh W).
\]
Therefore
\[
V=\Phi(W),
\qquad
\Phi(t)=-\sqrt{2\log(\cosh t)},
\qquad t<0.
\]
Direct differentiation gives
\[
\Phi'(t)
=
-\frac{\tanh t}{\sqrt{2\log(\cosh t)}}
>0,
\qquad t<0,
\]
and
\[
\Phi''(t)
=
\frac{
\tanh^{2}t
-
2\log(\cosh t)\operatorname{sech}^{2}t
}{
\bigl(2\log(\cosh t)\bigr)^{3/2}
}.
\]
After multiplying the numerator by $\cosh^{2}t$, it becomes
\[
Q(t)=\sinh^{2}t-2\log(\cosh t).
\]
The function $Q$ is even and satisfies $Q(0)=0$. Moreover, for
$t>0$,
\[
Q'(t)
=
2\sinh t\cosh t-2\tanh t
=
\frac{2\sinh^{3}t}{\cosh t}
>0.
\]
Hence
\[
Q(t)>0
\qquad\text{for all }t\neq0,
\]
and therefore
\[
\Phi''(t)>0
\qquad\text{for all }t<0.
\]

Since $V=\Phi(W)$, we obtain
\[
D^{2}V
=
\Phi'(W)D^{2}W
+
\Phi''(W)DW\otimes DW.
\]
For every nonzero $\xi\in\mathbb{R}^{N}$,
\[
\xi^{T}D^{2}V\xi
=
\Phi'(W)\xi^{T}D^{2}W\xi
+
\Phi''(W)(\xi\cdot DW)^{2}
>0.
\]
Thus $D^{2}V>0$, proving the assertion.
\end{remark}

The preceding argument is purely scalar. The arcosh transform is
adapted to the exponential source, while composition with the
increasing convex function $\Phi$ recovers the classical
square-root power-concavity conclusion.

We next show that the exponent $1/2$ is sharp for the Liouville
equation.

\begin{remark}
\label{rem:sharpness-half}
The exponent $1/2$ in the power-concavity consequence of
Theorem~1.1 is optimal, even within the class of smooth uniformly
strictly convex domains. More precisely, for every
\[
\alpha>\frac12,
\]
there exists a bounded, smooth, uniformly strictly convex domain
$\Omega_{\alpha}\subset\mathbb{R}^{N}$ such that the unique solution
of
\[
\begin{cases}
\Delta u=e^{u} & \text{in }\Omega_{\alpha},\\
u=0 & \text{on }\partial\Omega_{\alpha}
\end{cases}
\]
has the property that $(-u)^{\alpha}$ is not concave in
$\Omega_{\alpha}$.

Indeed, fix $\alpha>1/2$. By the sharpness result of Kennington
\cite[Theorem~6.2(iii)]{Ken85}, there exists a bounded convex domain
$\Omega\subset\mathbb{R}^{N}$ whose torsion function
\[
\begin{cases}
-\Delta U=1 & \text{in }\Omega,\\
U=0 & \text{on }\partial\Omega
\end{cases}
\]
has the property that $U^{\alpha}$ is not concave.

Since the failure of concavity is expressed by a strict inequality
at finitely many interior points, a standard approximation of
$\Omega$ by smooth uniformly strictly convex domains, together with
the stability of torsion functions under domain approximation,
allows us to assume that $\Omega$ itself is smooth and uniformly
strictly convex.

For $r>0$, let $u_{r}$ be the solution of
\[
\begin{cases}
\Delta u_{r}=e^{u_{r}} & \text{in }r\Omega,\\
u_{r}=0 & \text{on }\partial(r\Omega),
\end{cases}
\]
and define
\[
v_{r}(y)
=
-\frac{1}{r^{2}}u_{r}(ry),
\qquad y\in\Omega.
\]
Then $v_{r}>0$ and
\[
\begin{cases}
-\Delta v_{r}=e^{-r^{2}v_{r}} & \text{in }\Omega,\\
v_{r}=0 & \text{on }\partial\Omega.
\end{cases}
\]

Let
\[
M=\|U\|_{L^{\infty}(\Omega)}.
\]
Since
\[
e^{-r^{2}v_{r}}\leq1,
\]
comparison with the torsion function gives
\[
v_{r}\leq U
\qquad\text{in }\Omega.
\]
Consequently,
\[
e^{-r^{2}v_{r}}
\geq
e^{-r^{2}U}
\geq
e^{-r^{2}M}.
\]
A second application of the comparison principle yields
\[
e^{-r^{2}M}U
\leq
v_{r}
\leq
U
\qquad\text{in }\Omega.
\]
Therefore
\[
v_{r}\longrightarrow U
\qquad\text{uniformly in }\Omega
\]
as $r\to0$.

Because $U^{\alpha}$ is not concave, there exist
$x_{0},x_{1}\in\Omega$ and $\theta\in(0,1)$ such that, with
\[
x_{\theta}
=
(1-\theta)x_{0}+\theta x_{1},
\]
one has
\[
U(x_{\theta})^{\alpha}
<
(1-\theta)U(x_{0})^{\alpha}
+
\theta U(x_{1})^{\alpha}.
\]
The uniform convergence $v_{r}\to U$ implies that the same strict
inequality holds with $v_{r}$ in place of $U$ for all sufficiently
small $r>0$. Hence $v_{r}^{\alpha}$ is not concave in $\Omega$.

Finally,
\[
\bigl(-u_{r}(ry)\bigr)^{\alpha}
=
r^{2\alpha}v_{r}(y)^{\alpha}.
\]
Since multiplication by a positive constant and composition with a
linear dilation preserve concavity, $(-u_{r})^{\alpha}$ is not
concave in $r\Omega$ for all sufficiently small $r>0$.

Thus no exponent $\alpha>1/2$ gives a universal concavity result
for solutions of
\[
\Delta u=e^{u},
\qquad
u|_{\partial\Omega}=0.
\]
On the other hand, Theorem~1.1 and
Remark~\ref{rem:square-root-consequence} show that
$\sqrt{-u}$ is strictly concave. Therefore the exponent $1/2$ is
optimal.
\end{remark}

\section{Transformed Equations and Constant-Rank Structure}

\subsection{Scalar identities}

Let
\begin{equation}\label{eq:h-def}
h(w)=-2\log(\cosh w),\qquad u=h(w),\qquad w<0.
\end{equation}
Then
\begin{equation}\label{eq:h-derivatives}
h'(w)=-2\tanh w>0,\qquad
h''(w)=-2\sech^2w<0,\qquad
e^{h(w)}=\sech^2w.
\end{equation}
In particular,
\begin{equation}\label{eq:h-second-ode}
h''(w)=-2e^{h(w)}.
\end{equation}
Set
\begin{equation}\label{eq:s-def}
s(w)=-\sinh w\cosh w>0.
\end{equation}
Then
\begin{equation}\label{eq:h-identities}
\frac{h''}{h'}=-\frac1{s(w)},
\qquad
\frac{e^h}{h'}=\frac1{2s(w)},
\qquad
\frac{e^{2h}}{(h')^2}=\frac1{4s(w)^2}.
\end{equation}
Indeed,
\[
\frac{h''}{h'}=\frac{-2\sech^2w}{-2\tanh w}
=\frac1{\sinh w\cosh w}
=-\frac1{s(w)},
\]
and
\[
\frac{e^{2h}}{(h')^2}
=
\frac{\sech^4w}{4\tanh^2w}
=
\frac1{4\sinh^2w\cosh^2w}
=
\frac1{4s(w)^2}.
\]

The inverse relation is precisely
\begin{equation}\label{eq:h-inverse}
w=-\arcosh(e^{-u/2}).
\end{equation}
Moreover,
\begin{equation}\label{eq:s-convex}
s''(z)=-2\sinh(2z)>0,
\qquad z<0,
\end{equation}
so $s$ is strictly convex on $(-\infty,0)$.

\subsection{The transformed Liouville equation}

For $u=h(w)$, one has
\[
\Delta u=h'(w)\Delta w+h''(w)|Dw|^2.
\]
Thus \eqref{eq:laplace-main} and \eqref{eq:h-identities} give
\begin{equation}\label{eq:laplace-transformed}
F_L(D^2w,Dw,w)=0,
\end{equation}
where
\begin{equation}\label{eq:FL-def}
F_L(r,p,z)=s(z)\tr r-|p|^2-\frac12.
\end{equation}

We need one elementary inverse-convexity fact.

\begin{lemma}\label{lem:trace-inverse}
The map $A\mapsto\tr(A^{-1})$ is convex on $S^N_{++}$. For every
$\gamma>0$, the set
\[
K_\gamma=\{A\in S^N_{++}:\tr(A^{-1})\leq\gamma\}
\]
is convex and is closed under dilations $A\mapsto\lambda A$ with
$\lambda\geq1$.
\end{lemma}

\begin{proof}
For $A(t)=A+tH$,
\[
\left.\frac{d^2}{dt^2}\tr(A+tH)^{-1}\right|_{t=0}
=2\tr(A^{-1}HA^{-1}HA^{-1})\geq0.
\]
The last inequality follows after writing the matrix under the trace as
$A^{-1/2}(A^{-1/2}HA^{-1/2})^2A^{-1/2}$. Convexity of $K_\gamma$ follows,
and
\[
\tr((\lambda A)^{-1})=\lambda^{-1}\tr(A^{-1})
\]
proves dilation closure.
\end{proof}

\begin{proposition}\label{prop:laplace-constant-rank}
Let $D\subset\R^N$ be connected and let $w\in C^4(D)$ be a convex solution
of \eqref{eq:laplace-transformed}. Then $\rank D^2w$ is constant in $D$.
\end{proposition}

\begin{proof}
The equation is elliptic because
\[
F_L^{ij}=s(w)\delta_{ij}>0,
\]
and it is nondegenerate at the zero Hessian since
\[
F_L(0,p,z)=-|p|^2-\frac12\neq0.
\]
Fix $p\in\R^N$, put $\gamma_p=|p|^2+1/2$, and consider
\[
\mathcal C^L_p=
\{(A,z)\in S^N_{++}\times(-\infty,0):
F_L(A^{-1},p,z)\leq0\}.
\]
The defining inequality is
\[
s(z)\tr(A^{-1})\leq\gamma_p.
\]
With $B=A/s(z)$, this is equivalent to $B\in K_{\gamma_p}$. If
$(A_i,z_i)\in\mathcal C^L_p$, set $s_i=s(z_i)$ and $B_i=A_i/s_i$. For
$0\leq\theta\leq1$, let
\[
A_\theta=\theta A_1+(1-\theta)A_2,
\qquad z_\theta=\theta z_1+(1-\theta)z_2,
\qquad t_\theta=\theta s_1+(1-\theta)s_2.
\]
By \eqref{eq:s-convex}, $s(z_\theta)\leq t_\theta$, and
\[
\frac{A_\theta}{s(z_\theta)}
=\frac{t_\theta}{s(z_\theta)}
\left(
\frac{\theta s_1}{t_\theta}B_1
+\frac{(1-\theta)s_2}{t_\theta}B_2
\right).
\]
Lemma \ref{lem:trace-inverse} shows that this matrix belongs to
$K_{\gamma_p}$. Hence $\mathcal C^L_p$ is convex. All hypotheses of
Theorem \ref{thm:bian-guan} are satisfied, and the constant-rank conclusion
follows.
\end{proof}

\subsection{The transformed real \texorpdfstring{$\sigma_2$}{sigma2} equation}

Let $r=D^2w$ and $p=Dw$. Since
\[
D^2u=h'r+h''p\otimes p,
\]
we have, using \eqref{eq:real-rank-one},
\[
\sigma_2(D^2u)
=(h')^2\sigma_2(r)+h'h''\tr(P_{\R}(p)r).
\]
Thus $\sigma_2(D^2u)=e^{2u}=e^{2h}$ is equivalent to
\[
\sigma_2(r)+\frac{h''}{h'}\tr(P_{\R}(p)r)
=
\frac{e^{2h}}{(h')^2}.
\]
By \eqref{eq:h-identities},
\begin{equation}\label{eq:real-transformed}
\sigma_2(D^2w)
-\frac1{s(w)}\tr(P_{\R}(Dw)D^2w)
=
\frac1{4s(w)^2}.
\end{equation}
Equivalently,
\begin{equation}\label{eq:FR-equation}
F_{\R}(D^2w,Dw,w)=0,
\end{equation}
where
\begin{equation}\label{eq:FR-def}
F_{\R}(r,p,z)
=
s(z)^2\sigma_2(r)
-s(z)\tr(P_{\R}(p)r)
-\frac14.
\end{equation}

\paragraph{Ellipticity and nondegeneracy.}
We next verify that the transformed operator $F_R$ is elliptic along
admissible solutions and is nondegenerate at the zero Hessian, as required
for the application of the Bian–Guan constant-rank theorem.
Define
\begin{equation}\label{eq:real-N}
N=r-\frac1{s(z)}p\otimes p.
\end{equation}
Along the solution,
\begin{equation}\label{eq:D2u-real-N}
D^2u=h'(w)N.
\end{equation}
Since $h'(w)>0$ and $D^2u\in\Gamma_2$, we have $N\in\Gamma_2$. Differentiating $F_{\R}$ in the $r$ variable gives
\[
F_{\R}^{ij}
=
s^2T_1(r)^{ij}-sP_{\R}(p)^{ij}.
\]
Using
\[
T_1\left(r-\frac1s p\otimes p\right)
=
T_1(r)-\frac1sP_{\R}(p),
\]
we obtain
\begin{equation}\label{eq:real-ellipticity}
F_{\R}^{ij}=s^2T_1(N)^{ij}>0.
\end{equation}
Thus the transformed equation is elliptic along admissible solutions. Also
\begin{equation}\label{eq:real-nondegenerate}
F_{\R}(0,p,z)=-\frac14\neq0.
\end{equation}

\paragraph{Level-set convexity.}

Fix $p\in\R^n$ and consider
\begin{equation}\label{eq:real-level-set}
\mathcal C^{\R}_p
=
\{(A,z)\in S^n_{++}\times(-\infty,0):
F_{\R}(A^{-1},p,z)\leq0\}.
\end{equation}
The inequality is
\begin{equation}\label{eq:real-level-ineq}
s(z)^2\sigma_2(A^{-1})
-s(z)\tr(P_{\R}(p)A^{-1})
\leq\frac14.
\end{equation}
Set
\[
B=\frac{A}{s(z)}.
\]
Then $A^{-1}=s(z)^{-1}B^{-1}$, and \eqref{eq:real-level-ineq} becomes
\begin{equation}\label{eq:real-B-ineq}
\sigma_2(B^{-1})-\tr(P_{\R}(p)B^{-1})\leq\frac14.
\end{equation}
Let
\begin{equation}\label{eq:real-K}
K^{\R}_p
=
\left\{B\in S^n_{++}:
\sigma_2(B^{-1})-\tr(P_{\R}(p)B^{-1})\leq\frac14
\right\}.
\end{equation}
We claim that $K^{\R}_p$ is convex. If $p\neq0$, write $p=|p|\alpha$ and $Q_\alpha=I-\alpha\otimes\alpha$. Since
\[
P_{\R}(p)=|p|^2Q_\alpha
\]
and $\tr(Q_\alpha B^{-1})>0$, condition \eqref{eq:real-B-ineq} is equivalent to
\[
\frac{\sigma_2(B^{-1})-\frac14}
{\tr(Q_\alpha B^{-1})}
\leq |p|^2.
\]
By Proposition \ref{prop:real-inverse}, this is a sublevel set of a convex function. Hence $K^{\R}_p$ is convex.

If $p=0$, fix a unit vector $\alpha$ and set $p_\varepsilon=\varepsilon\alpha$. Then $K^{\R}_{p_\varepsilon}$ is convex for all $\varepsilon>0$, and
\[
K^{\R}_0=\bigcap_{\varepsilon>0}K^{\R}_{p_\varepsilon}.
\]
Indeed,
\[
B\in K^{\R}_{p_\varepsilon}
\quad\Longleftrightarrow\quad
\sigma_2(B^{-1})-\varepsilon^2\tr(Q_\alpha B^{-1})\leq\frac14,
\]
and letting $\varepsilon\downarrow0$ gives the claim. Thus $K^{\R}_0$ is convex.

Next, $K^{\R}_p$ is closed under dilations $B\mapsto\lambda B$, $\lambda\geq1$. For $B\in K^{\R}_p$,
\begin{align*}
&\sigma_2((\lambda B)^{-1})
-\tr(P_{\R}(p)(\lambda B)^{-1})\\
&\qquad=
\lambda^{-2}\sigma_2(B^{-1})
-\lambda^{-1}\tr(P_{\R}(p)B^{-1})\\
&\qquad=
\lambda^{-2}\{\sigma_2(B^{-1})-\tr(P_{\R}(p)B^{-1})\}
+(\lambda^{-2}-\lambda^{-1})\tr(P_{\R}(p)B^{-1}).
\end{align*}
Because $P_{\R}(p)\geq0$, $B^{-1}>0$, and $\lambda^{-2}-\lambda^{-1}\leq0$, this is at most $\lambda^{-2}/4\leq1/4$. Therefore $\lambda B\in K^{\R}_p$.

We now prove that $\mathcal C^{\R}_p$ is convex. Take $(A_i,z_i)\in\mathcal C^{\R}_p$, $i=1,2$, and $0\leq\theta\leq1$. Put
\[
s_i=s(z_i),\qquad B_i=\frac{A_i}{s_i}\in K^{\R}_p.
\]
Let
\[
A_\theta=\theta A_1+(1-\theta)A_2,
\qquad
z_\theta=\theta z_1+(1-\theta)z_2.
\]
Since
\[
s(z)=-\sinh z\cosh z=-\frac12\sinh(2z),
\]
one has
\[
s''(z)=-2\sinh(2z)>0,\qquad z<0,
\]
so $s$ is convex on $(-\infty,0)$. Thus
\[
s_\theta:=s(z_\theta)
\leq
t_\theta:=\theta s_1+(1-\theta)s_2.
\]
Moreover,
\[
A_\theta
=
t_\theta
\left(
\frac{\theta s_1}{t_\theta}B_1
+\frac{(1-\theta)s_2}{t_\theta}B_2
\right).
\]
The matrix in parentheses belongs to $K^{\R}_p$ by convexity, and multiplication by $t_\theta/s_\theta\geq1$ keeps it in $K^{\R}_p$ by dilation closure. Therefore
\[
\frac{A_\theta}{s_\theta}\in K^{\R}_p,
\]
hence $(A_\theta,z_\theta)\in\mathcal C^{\R}_p$. This proves the level-set convexity.

\begin{proposition}\label{prop:real-constant-rank}
Let $D\subset\R^n$, $n\geq3$, be connected. Let $w\in C^4(D)$ solve \eqref{eq:FR-equation}. Assume
\[
D^2w\geq0
\]
and
\[
D^2w-\frac1{s(w)}Dw\otimes Dw\in\Gamma_2.
\]
Then $\rank D^2w$ is constant in $D$.
\end{proposition}

\begin{proof}
The ellipticity, nondegeneracy, and level-set convexity conditions required by Theorem \ref{thm:bian-guan} were verified in \eqref{eq:real-ellipticity}, \eqref{eq:real-nondegenerate}, and the preceding discussion. Since $D^2w\geq0$, Theorem \ref{thm:bian-guan} applies and gives the conclusion.
\end{proof}

\subsection{The Complex Transformed Equation and Constant-Rank Structure}

Let $\Omega\subset\C^m$ and let $u=h(w)$ with $h$ as in \eqref{eq:h-def}. Put
\[
M=(w_{i\bar j}),\qquad q=\partial w=(w_i)_{i=1}^m.
\]
Then
\begin{equation}\label{eq:complex-uij}
u_{i\bar j}
=
h'w_{i\bar j}+h''w_iw_{\bar j}
=
h'\left(w_{i\bar j}-\frac1{s(w)}w_iw_{\bar j}\right).
\end{equation}
Thus the complex equation $\sigma_2(u_{i\bar j})=e^{2u}$ is equivalent to
\begin{equation}\label{eq:complex-transformed}
\sigma_2(M)-\frac1{s(w)}\tr(P_{\C}(q)M)
=
\frac1{4s(w)^2}.
\end{equation}
Equivalently,
\begin{equation}\label{eq:complex-transformed-2}
s(w)^2\sigma_2(w_{i\bar j})
-s(w)\tr(P_{\C}(\partial w)(w_{i\bar j}))
-\frac14=0.
\end{equation}

In real notation, let
\[
H=D^2_{\R}w,\qquad p=D_{\R}w,
\]
and let $q(p)=\partial w$. Define
\begin{equation}\label{eq:FC-def}
F_{\C}(H,p,z)
=
s(z)^2\sigma_2(C(H))
-s(z)\tr(P_{\C}(q(p))C(H))
-\frac14.
\end{equation}
Then \eqref{eq:complex-transformed-2} is
\[
F_{\C}(D^2_{\R}w,D_{\R}w,w)=0.
\]

\paragraph{Ellipticity and nondegeneracy.}

Along the solution define
\begin{equation}\label{eq:complex-N}
N=C(D^2_{\R}w)-\frac1{s(w)}\partial w\otimes\overline{\partial w}.
\end{equation}
By \eqref{eq:complex-uij},
\[
(u_{i\bar j})=h'(w)N.
\]
Since $h'(w)>0$ and $(u_{i\bar j})\in\Gamma_2$, we have $N\in\Gamma_2$. Let $\lambda\in\R^{2m}\setminus\{0\}$ and let $\zeta\in\C^m$ be its associated complex vector. Differentiating $F_{\C}$ in the direction $\lambda\otimes\lambda$ gives
\begin{equation}\label{eq:complex-ellipticity}
D_HF_{\C}[\lambda\otimes\lambda]
=
s(w)^2\tr\bigl(T_1(N)C(\lambda\otimes\lambda)\bigr)
=
\frac{s(w)^2}{4}\,\overline{\zeta}^{\,T}T_1(N)\zeta
>0.
\end{equation}
Thus $F_{\C}$ is strictly elliptic in the real Hessian variables along the solution. Moreover,
\begin{equation}\label{eq:complex-nondegenerate}
F_{\C}(0,p,z)=-\frac14\neq0.
\end{equation}

\paragraph{Level-set convexity.}

Fix $p\in\R^{2m}$, and write $q=q(p)$. Consider
\begin{equation}\label{eq:complex-level-set}
\mathcal C^{\C}_p
=
\{(A,z)\in S^{2m}_{++}\times(-\infty,0):
F_{\C}(A^{-1},p,z)\leq0\}.
\end{equation}
The sublevel inequality is
\begin{equation}\label{eq:complex-level-ineq}
s(z)^2\sigma_2(C(A^{-1}))
-s(z)\tr(P_{\C}(q)C(A^{-1}))
\leq\frac14.
\end{equation}
Set $B=A/s(z)$. Since $C$ is linear,
\[
C(A^{-1})=\frac1{s(z)}C(B^{-1}).
\]
Therefore \eqref{eq:complex-level-ineq} is equivalent to
\begin{equation}\label{eq:complex-B-ineq}
\sigma_2(C(B^{-1}))
-\tr(P_{\C}(q)C(B^{-1}))
\leq\frac14.
\end{equation}
By the Hermitian rank-one identity \eqref{eq:complex-rank-one}, this is
\begin{equation}\label{eq:complex-K-ineq}
\sigma_2(C(B^{-1})-q\otimes\bar q)\leq\frac14.
\end{equation}
Thus define
\[
K^{\C}_q
=
\left\{
B\in S^{2m}_{++}:
\sigma_2(C(B^{-1})-q\otimes\bar q)\leq\frac14
\right\}.
\]
By Proposition \ref{prop:complex-inverse}, $K^{\C}_q$ is convex. It remains to check dilation closure. If $B\in K^{\C}_q$ and $\lambda\geq1$, let $C_0=C(B^{-1})$. Since
\[
C((\lambda B)^{-1})=\lambda^{-1}C_0,
\]
we have
\begin{align*}
\sigma_2(\lambda^{-1}C_0-q\otimes\bar q)
&=
\lambda^{-2}\sigma_2(C_0)
-\lambda^{-1}\tr(P_{\C}(q)C_0)\\
&=
\lambda^{-2}\{\sigma_2(C_0)-\tr(P_{\C}(q)C_0)\}\\
&\quad
+(\lambda^{-2}-\lambda^{-1})\tr(P_{\C}(q)C_0).
\end{align*}
Since $P_{\C}(q)\geq0$ and $C_0>0$, the trace term is nonnegative, while $\lambda^{-2}-\lambda^{-1}\leq0$. Hence the last expression is at most $\lambda^{-2}/4\leq1/4$. Thus $\lambda B\in K^{\C}_q$.

Now the same scaling argument as in the real case proves that $\mathcal C^{\C}_p$ is convex. If $(A_i,z_i)\in\mathcal C^{\C}_p$, let $s_i=s(z_i)$ and $B_i=A_i/s_i\in K^{\C}_q$. For a convex combination $(A_\theta,z_\theta)$, the convexity of $s$ gives
\[
s(z_\theta)\leq\theta s_1+(1-\theta)s_2.
\]
Hence $A_\theta/s(z_\theta)$ is a dilation by a factor at least one of a convex combination of $B_1$ and $B_2$, and therefore belongs to $K^{\C}_q$. Thus $(A_\theta,z_\theta)\in\mathcal C^{\C}_p$.

\begin{proposition}\label{prop:complex-constant-rank}
Let $D\subset\C^m$, $m\geq2$, be connected. Let $w\in C^4(D)$ solve \eqref{eq:complex-transformed-2}. Assume
\[
D^2_{\R}w\geq0
\]
and
\[
C(D^2_{\R}w)-\frac1{s(w)}\partial w\otimes\overline{\partial w}\in\Gamma_2.
\]
Then $\rank D^2_{\R}w$ is constant in $D$.
\end{proposition}

\begin{proof}
The equation is written as
\[
F_{\C}(D^2_{\R}w,D_{\R}w,w)=0.
\]
Ellipticity follows from \eqref{eq:complex-ellipticity}, nondegeneracy from \eqref{eq:complex-nondegenerate}, and the level-set condition from the preceding discussion. Therefore Theorem \ref{thm:bian-guan} gives constant rank.
\end{proof}

\section{Boundary Strict Convexity and Ball Models}

Now we state the following well-known boundary convexity lemma; see, for example, Caffarelli and Friedman \cite{CaffarelliFriedman1985} or Korevaar \cite{Korevaar1983}.In this section, uniformly convex means that the second fundamental form of
$\partial\Omega$ with respect to the exterior unit normal is uniformly
positive definite.
\begin{lemma}\label{lem:boundary-strict-convexity}
Let $\Omega\subset\mathbb R^N$ be a bounded, smooth, and uniformly
strictly convex domain. Suppose that
\[
u\in C^2(\overline{\Omega}),
\qquad
u<0\quad\hbox{in }\Omega,
\qquad
u=0\quad\hbox{on }\partial\Omega,
\qquad
\partial_\nu u>0\quad\hbox{on }\partial\Omega,
\]
where $\nu$ denotes the exterior unit normal.

Let
\[
v=f(u),
\]
where
\[
f\in C^2((-\infty,0))\cap C^0((-\infty,0]),
\qquad
f'>0,
\qquad
f''>0,
\]
and assume that
\[
\lim_{q\to0^-}\frac{f'(q)}{f''(q)}=0.
\]
Then
\[
v\in C^2(\Omega)\cap C^0(\overline{\Omega}),
\]
and there exists $\delta>0$ such that
\[
D^2v>0
\quad\hbox{in}\quad
\Omega_\delta
:=
\left\{
x\in\Omega:
0<\operatorname{dist}(x,\partial\Omega)<\delta
\right\}.
\]
\end{lemma}
We apply this lemma to
\begin{equation}\label{eq:f-def}
f(u)=-\arcosh(e^{-u/2}).
\end{equation}
Since $u=h(w)$ and $w=f(u)$, one has
\[
f'(u)=\frac1{h'(w)}>0,
\qquad
f''(u)=-\frac{h''(w)}{(h'(w))^3}>0.
\]
Moreover,
\begin{equation}\label{eq:f-ratio}
\frac{f'}{f''}
=
-\frac{(h')^2}{h''}
=
2\sinh^2w
\to0
\qquad\hbox{as }u\to0^-.
\end{equation}

\medskip
\noindent\emph{Boundary behavior of the transformed function.}
Although the original solution is smooth up to the boundary, the
transformed function is generally not differentiable up to the
boundary. Indeed, let
\[
d(x):=\operatorname{dist}(x,\partial\Omega),
\]
and let $\pi(x)$ denote the nearest boundary point. Since
$u\in C^2(\overline{\Omega})$, $u=0$ on $\partial\Omega$, and
$\partial_\nu u>0$ on $\partial\Omega$, one has
\[
u(x)
=
-\partial_\nu u(\pi(x))\,d(x)
+
O(d(x)^2)
\qquad\hbox{as }d(x)\to0.
\]
Moreover,
\[
-\operatorname{arcosh}(e^{-q/2})
=
-\sqrt{-q}
+
O((-q)^{3/2})
\qquad\hbox{as }q\to0^-.
\]
Consequently,
\[
w(x)
=
-\sqrt{\partial_\nu u(\pi(x))\,d(x)}
+
O(d(x)^{3/2}).
\]
Thus
\[
w\in C^\infty(\Omega)\cap C^0(\overline{\Omega}),
\]
but in general
\[
w\notin C^1(\overline{\Omega}).
\]
Accordingly, boundary strict convexity means strict convexity in an
interior boundary collar $\Omega_\delta$, rather than the existence
of a finite Hessian of $w$ on $\partial\Omega$.

Thus Lemma \ref{lem:boundary-strict} applies once the Hopf condition is
known. For the Liouville problem, $\Delta u=e^u>0$. In the real Hessian
case, admissibility gives
\[
\Delta u=\sigma_1(D^2u)>0.
\]
In the complex case, admissibility gives
\[
\Delta_{\R}u=4\sum_{j=1}^m u_{j\bar j}
=4\sigma_1(u_{i\bar j})>0.
\]
Together with $u<0$ in $\Omega$ and $u=0$ on $\partial\Omega$, the Hopf
boundary point lemma gives $\partial_\nu u>0$. Hence, in all three settings,
the corresponding transform $w=f(u)$ is strictly convex in the real
variables in a boundary strip.


\subsection{The Liouville ball}

\begin{lemma}\label{lem:laplace-ball}
Let $B_R\subset\R^N$, $N\geq2$, be a ball, and let $u$ solve
\eqref{eq:laplace-main} in $B_R$. Then
\[
w=-\arcosh(e^{-u/2})
\]
satisfies $D^2w>0$ in $B_R$.
\end{lemma}

\begin{proof}
Uniqueness and rotational invariance imply that $u(x)=\phi(r)$, where
$r=|x|$. The equation becomes
\[
\phi''+\frac{N-1}{r}\phi'=e^\phi,
\]
and hence
\[
(r^{N-1}\phi')'=r^{N-1}e^\phi>0.
\]
Since $\phi'(0)=0$, it follows that $\phi'(r)>0$ for $0<r<R$.

Write $w(x)=\psi(r)$ and $y=\psi'$. The scalar map
$q\mapsto-\arcosh(e^{-q/2})$ is strictly increasing on $(-\infty,0)$, so
\begin{equation}\label{eq:laplace-y-positive}
y(r)>0,
\qquad 0<r<R.
\end{equation}
For a radial function,
\[
\Delta w=y'+\frac{N-1}{r}y,
\qquad |Dw|^2=y^2.
\]
Equation \eqref{eq:laplace-transformed} therefore reads
\begin{equation}\label{eq:laplace-ball-ode}
s(\psi)\left(y'+\frac{N-1}{r}y\right)-y^2=\frac12.
\end{equation}

At the origin, let
\[
L=\lim_{r\to0^+}\frac{y(r)}r=\psi''(0).
\]
The leading terms in \eqref{eq:laplace-ball-ode} give
\[
Ns(\psi(0))L=\frac12,
\]
so $L>0$. Suppose that $y'$ has a first zero $r_0\in(0,R)$. Then
$y'>0$ on $(0,r_0)$, while $y'(r_0)=0$ and $y''(r_0)\leq0$. Differentiating
\eqref{eq:laplace-ball-ode} and evaluating at $r_0$ yields
\[
s(\psi)y''
=(N-1)\frac{s(\psi)y}{r^2}
-(N-1)\frac{s'(\psi)y^2}{r}.
\]
Here $s>0$, $y>0$, and
\[
s'(z)=-(\cosh^2z+\sinh^2z)<0.
\]
The right-hand side is therefore strictly positive, contradicting
$y''(r_0)\leq0$. Thus $y'>0$ on $(0,R)$. The eigenvalues of $D^2w$ are
$y'$ in the radial direction and $y/r$ in the tangential directions, and
both are positive.
\end{proof}

\subsection{The real ball}

\begin{lemma}\label{lem:real-ball}
Let $B_R\subset\R^n$, $n\geq3$, be a ball. Let $u$ be the admissible solution of \eqref{eq:real-main} in $B_R$. Then
\[
w=-\arcosh(e^{-u/2})
\]
satisfies $D^2w>0$ in $B_R$.
\end{lemma}

\begin{proof}
By uniqueness of the Dirichlet problem and rotational invariance, $u$ is radial:
\[
u(x)=\phi(r),\qquad r=|x|.
\]
Since $\Delta u>0$,
\[
(r^{n-1}\phi'(r))'>0.
\]
Together with $\phi'(0)=0$, this gives
\begin{equation}\label{eq:real-radial-u-positive}
\phi'(r)>0,\qquad 0<r<R.
\end{equation}
Write $w(x)=\psi(r)$ and set
\[
y(r)=\psi'(r).
\]
Since $w=f(u)$ and $f'>0$, \eqref{eq:real-radial-u-positive} implies
\begin{equation}\label{eq:y-positive-real}
y(r)>0,\qquad 0<r<R.
\end{equation}
The eigenvalues of $D^2w$ are $y'(r)$ and $y(r)/r$ with multiplicity $n-1$. Therefore
\begin{equation}\label{eq:radial-sigma-real}
\sigma_2(D^2w)
=
(n-1)\frac{yy'}r
+
\frac{(n-1)(n-2)}2\frac{y^2}{r^2}.
\end{equation}
Moreover $Dw=ye_r$, so $P_{\R}(Dw)$ vanishes in the radial direction and equals $y^2$ in tangential directions. Thus
\begin{equation}\label{eq:radial-trace-real}
\tr(P_{\R}(Dw)D^2w)
=(n-1)\frac{y^3}{r}.
\end{equation}
Substituting \eqref{eq:radial-sigma-real} and \eqref{eq:radial-trace-real} into \eqref{eq:FR-equation}, and multiplying by $r^2/(n-1)$, gives
\begin{equation}\label{eq:real-radial-ode}
s^2 r y y'+c s^2 y^2-s r y^3
=
\frac{r^2}{4(n-1)},
\qquad
c=\frac{n-2}{2}>0,
\end{equation}
where $s=s(\psi(r))$.

We prove $y'>0$. At $r=0$, let
\[
L=\lim_{r\to0^+}\frac{y(r)}r=\psi''(0),
\qquad
s_0=s(\psi(0))>0.
\]
Using $y(r)=Lr+o(r)$ in \eqref{eq:real-radial-ode}, the leading $r^2$ terms give
\[
s_0^2(1+c)L^2=\frac1{4(n-1)}.
\]
Since $1+c=n/2$,
\begin{equation}\label{eq:L-real}
L^2=\frac1{2n(n-1)s_0^2}.
\end{equation}
Because $y(r)>0$ for small $r>0$, $L>0$. Hence $\psi''(0)>0$.

Assume by contradiction that $y'$ has a first zero $r_0\in(0,R)$. Then
\[
y'>0\quad\hbox{on }[0,r_0),\qquad
y'(r_0)=0,\qquad
y''(r_0)\leq0.
\]
At $r_0$, equation \eqref{eq:real-radial-ode} gives
\[
c s^2 y^2-s r y^3=\frac{r^2}{4(n-1)}>0,
\]
and hence
\begin{equation}\label{eq:ry-cs-real}
ry<cs
\qquad\hbox{at }r=r_0.
\end{equation}
Differentiate \eqref{eq:real-radial-ode}. Since
\[
s'(z)=-(\cosh^2z+\sinh^2z)<0
\]
and $d(s(\psi(r)))/dr=s'(\psi)y$, evaluation at $r_0$, where $y'=0$, gives
\[
2c s s'y^3+s^2 r y y''-s' r y^4-sy^3
=
\frac{r}{2(n-1)}.
\]
Thus
\begin{equation}\label{eq:ypp-real}
s^2 r y y''
=
\frac{r}{2(n-1)}
+sy^3
+s'y^3(ry-2cs).
\end{equation}
By \eqref{eq:ry-cs-real}, $ry-2cs<0$, and since $s'<0$, the last term is positive. Therefore the right-hand side of \eqref{eq:ypp-real} is positive, so $y''(r_0)>0$, a contradiction. Hence $y'(r)>0$ on $(0,R)$. Since $y/r>0$, all eigenvalues of $D^2w$ are positive.
\end{proof}

\subsection{The complex ball}

\begin{lemma}\label{lem:complex-ball}
Let $B_R\subset\C^m$, $m\geq2$, be a Euclidean ball. Let $u$ be the complex admissible solution of \eqref{eq:complex-main} in $B_R$. Then
\[
w=-\arcosh(e^{-u/2})
\]
satisfies $D^2_{\R}w>0$ in $B_R$.
\end{lemma}

\begin{proof}
By uniqueness and unitary invariance, $u$ is radial:
\[
u(z)=\phi(r),\qquad r=|z|.
\]
Complex admissibility gives
\[
\Delta_{\R}u=4\sigma_1(u_{i\bar j})>0.
\]
Therefore
\[
(r^{2m-1}\phi'(r))'>0,
\]
and since $\phi'(0)=0$,
\begin{equation}\label{eq:complex-u-radial-positive}
\phi'(r)>0,\qquad 0<r<R.
\end{equation}
Let $w(z)=\psi(r)$ and $y=\psi'$. Again $f'>0$ gives
\begin{equation}\label{eq:complex-y-positive}
y(r)>0,\qquad 0<r<R.
\end{equation}
For a radial function, the complex Hessian has one radial eigenvalue
\begin{equation}\label{eq:complex-radial-a}
a=\frac14\left(y'+\frac yr\right)
\end{equation}
and $m-1$ tangential eigenvalues
\begin{equation}\label{eq:complex-radial-b}
b=\frac{y}{2r}.
\end{equation}
Hence
\begin{equation}\label{eq:complex-radial-sigma}
\sigma_2(w_{i\bar j})
=
\frac{m-1}{8}\left(\frac{yy'}r+\frac{(m-1)y^2}{r^2}\right).
\end{equation}
Moreover $|\partial w|^2=y^2/4$, and $P_{\C}(\partial w)$ vanishes in the complex radial direction and equals $y^2/4$ in the complex tangential directions. Therefore
\begin{equation}\label{eq:complex-radial-trace}
\tr(P_{\C}(\partial w)(w_{i\bar j}))
=
\frac{m-1}{8}\frac{y^3}{r}.
\end{equation}
Substituting \eqref{eq:complex-radial-sigma} and \eqref{eq:complex-radial-trace} into \eqref{eq:complex-transformed-2} and multiplying by $8r^2/(m-1)$ gives
\begin{equation}\label{eq:complex-radial-ode}
s^2 r y y'+c s^2 y^2-s r y^3=dr^2,
\qquad
c=m-1>0,\quad d=\frac2{m-1}>0.
\end{equation}
At the origin, let
\[
L=\lim_{r\to0^+}\frac{y(r)}r,\qquad s_0=s(\psi(0))>0.
\]
The leading terms in \eqref{eq:complex-radial-ode} give
\[
s_0^2(1+c)L^2=d.
\]
Since $1+c=m$, we obtain
\begin{equation}\label{eq:L-complex}
L^2=\frac{2}{m(m-1)s_0^2}.
\end{equation}
The positivity of $y$ near the origin implies $L>0$.

Suppose $y'$ has a first zero $r_0\in(0,R)$. Then $y'(r_0)=0$ and $y''(r_0)\leq0$. Evaluating \eqref{eq:complex-radial-ode} at $r_0$ gives
\[
c s^2 y^2-s r y^3=dr^2>0,
\]
so
\begin{equation}\label{eq:ry-cs-complex}
ry<cs.
\end{equation}
Differentiating \eqref{eq:complex-radial-ode} and evaluating at $r_0$ yields
\[
2c s s'y^3+s^2 r y y''-s' r y^4-sy^3=2dr.
\]
Thus
\begin{equation}\label{eq:ypp-complex}
s^2 r y y''
=
2dr+sy^3+s'y^3(ry-2cs).
\end{equation}
Since $s'<0$ and $ry-2cs<0$, the last term is positive. Hence $s^2 r y y''>0$ and $y''(r_0)>0$, a contradiction. Therefore $y'>0$ on $(0,R)$.

The real Hessian of a radial real function on $\R^{2m}$ has eigenvalues $\psi''=y'$ in the radial direction and $\psi'/r=y/r$ in tangential real directions. Both are positive, so $D^2_{\R}w>0$ in the ball.
\end{proof}

\section{Domain Deformation and Proofs of the Main Theorems}
\subsection{Solvability,Boundary Regularity, and Uniqueness along the Deformation}
\label{subsec:solvability-uniqueness}
\paragraph{Solvability and boundary regularity.}
We first explain why the solutions used below are smooth up to the
boundary. Since the domains $\Omega_t$ form a smooth family of
uniformly strictly convex domains with uniformly controlled geometry,
one may choose smooth defining functions $\rho_t$ such that
\[
\rho_t<0\quad\hbox{in }\Omega_t,
\qquad
\rho_t=0\quad\hbox{on }\partial\Omega_t,
\]
and
\[
D^2\rho_t\geq \kappa I
\quad\hbox{on }\overline{\Omega_t}
\]
for some constant $\kappa>0$ independent of $t$.

For the Liouville equation, if $A>0$ is sufficiently large, then
\[
\Delta(A\rho_t)
=
A\Delta\rho_t
\geq 1
\geq e^{A\rho_t}
\quad\hbox{in }\Omega_t.
\]
Thus $A\rho_t$ is a subsolution, while the zero function is a
supersolution, since
\[
\Delta 0=0\leq 1=e^0.
\]
The sub- and supersolution method therefore gives a solution
$u_t$ satisfying
\[
A\rho_t\leq u_t\leq 0.
\]

For the real $\sigma_2$ equation, the strict convexity of $\rho_t$
gives
\[
D^2\rho_t\in\Gamma_2.
\]
Moreover,
\[
\sigma_2(D^2(A\rho_t))
=
A^2\sigma_2(D^2\rho_t).
\]
Hence, after increasing $A$ if necessary,
\[
\sigma_2(D^2(A\rho_t))
\geq 1
\geq e^{2A\rho_t}.
\]
Thus $A\rho_t$ is an admissible subsolution and the zero function is
a supersolution. The admissible Dirichlet theory of
Caffarelli--Nirenberg--Spruck \cite{CaffarelliNirenbergSpruck1985} therefore gives an
admissible solution.

In the complex case, real strict convexity implies that $\rho_t$ is
strictly plurisubharmonic. Indeed, by the compression formula,
\[
\bigl((\rho_t)_{i\bar j}\bigr)
=
\mathcal C(D_{\mathbb R}^2\rho_t)>0.
\]
Consequently,
\[
\sigma_2\bigl((A\rho_t)_{i\bar j}\bigr)
=
A^2\sigma_2\bigl((\rho_t)_{i\bar j}\bigr)
\geq 1
\geq e^{2A\rho_t}
\]
for $A$ sufficiently large. Hence $A\rho_t$ is a complex admissible
subsolution, while zero is a supersolution. The complex Hessian
Dirichlet theory, together with its global boundary estimates,
then gives a complex admissible solution; see, for example,
\cite{CollinsPicard2022}.

In all three cases, the right-hand side is smooth and is increasing
in $u$. The global Dirichlet estimates first give
\[
u_t\in C^{2,\alpha}(\overline{\Omega_t})
\]
for some $\alpha\in(0,1)$. Once this estimate is known, the equations
are uniformly elliptic along the admissible solution. Standard
Schauder bootstrapping then gives
\[
u_t\in C^\infty(\overline{\Omega_t}).
\]
Thus the boundary regularity asserted in
Theorems~1.1--1.3 is automatic and is not an additional regularity
assumption.
For each fixed \(t\in[0,1]\), the solution \(u_t\) used below is unique
in the corresponding admissible class. Indeed, let \(u\) and \(v\) be
two solutions on \(\Omega_t\) with the same zero boundary data. Suppose
that \(u-v\) has a positive maximum at some point
\(x_{*}\in\Omega_t\). Then
\[
u(x_{*})>v(x_{*}),
\qquad
D_{\mathbb R}^{2}u(x_{*})
\leq
D_{\mathbb R}^{2}v(x_{*}).
\]

For the Liouville equation, this gives
\[
e^{u(x_{*})}
=
\Delta u(x_{*})
\leq
\Delta v(x_{*})
=
e^{v(x_{*})},
\]
which contradicts \(u(x_{*})>v(x_{*})\).

For the real \(\sigma_{2}\) equation, the monotonicity of
\(\sigma_{2}\) on the admissible cone, recalled in Subsection~2.1,
gives
\[
e^{2u(x_{*})}
=
\sigma_{2}\bigl(D^{2}u(x_{*})\bigr)
\leq
\sigma_{2}\bigl(D^{2}v(x_{*})\bigr)
=
e^{2v(x_{*})},
\]
again a contradiction.

For the complex \(\sigma_{2}\) equation, formula~\((16)\) shows that
the compression map \(C\) preserves the semidefinite order. Hence
\[
\bigl(u_{i\bar j}(x_{*})\bigr)
=
C\bigl(D_{\mathbb R}^{2}u(x_{*})\bigr)
\leq
C\bigl(D_{\mathbb R}^{2}v(x_{*})\bigr)
=
\bigl(v_{i\bar j}(x_{*})\bigr).
\]
The monotonicity of \(\sigma_{2}\) on the complex admissible cone
therefore yields
\[
e^{2u(x_{*})}
=
\sigma_{2}\bigl(u_{i\bar j}(x_{*})\bigr)
\leq
\sigma_{2}\bigl(v_{i\bar j}(x_{*})\bigr)
=
e^{2v(x_{*})},
\]
which is impossible.

Thus \(u\leq v\). Interchanging \(u\) and \(v\) gives \(v\leq u\),
and consequently \(u=v\). Therefore, for every \(t\in[0,1]\), the
notation \(u_t\) below is unambiguous.
\medskip
\subsection{Local \texorpdfstring{$C^2$}{C2} Stability under Domain Deformations}

We first record the semilinear stability statement used for the Liouville
equation.

\begin{proposition}\label{prop:liouville-stability}
Let $\alpha\in(0,1)$, and let
$\{\Omega_t\}_{t\in[0,1]}$ be a smooth family of bounded smooth
uniformly strictly convex domains with uniformly controlled geometry.
Let $u_t$ be the unique solution of
\[
\Delta u_t=e^{u_t}
\quad\hbox{in }\Omega_t,
\qquad
u_t=0
\quad\hbox{on }\partial\Omega_t.
\]
Fix $t_0\in[0,1]$. Then there exist smooth diffeomorphisms
\[
\Phi_t:\Omega_{t_0}\longrightarrow\Omega_t,
\qquad
\Phi_{t_0}=\operatorname{id},
\]
such that
\[
u_t\circ\Phi_t
\longrightarrow
u_{t_0}
\quad\hbox{in }
C^{2,\alpha}(\overline{\Omega_{t_0}})
\quad\hbox{as }t\to t_0.
\]
Consequently, if
\[
w_t=-\operatorname{arcosh}(e^{-u_t/2}),
\]
then, for every $K\Subset\Omega_{t_0}$,
\[
w_t\circ\Phi_t
\longrightarrow
w_{t_0}
\quad\hbox{in }C^{2,\alpha}(K).
\]
\end{proposition}

\begin{proof}
Write
\[
\Omega_0:=\Omega_{t_0},
\qquad
u_0:=u_{t_0}.
\]
Since the domains form a smooth family, after restricting $t$ to a
sufficiently small neighborhood of $t_0$, there exists a smooth
family of diffeomorphisms
\[
\Phi_t:\Omega_0\longrightarrow\Omega_t,
\qquad
\Phi_{t_0}=\operatorname{id}.
\]
Let
\[
\Psi_t:=\Phi_t^{-1}.
\]
For $v\in C^{2,\alpha}(\overline{\Omega_0})$, define the pulled-back
Laplacian by
\[
\Delta_t v(x)
:=
\left.
\Delta_y(v\circ\Psi_t)(y)
\right|_{y=\Phi_t(x)}.
\]
Its coefficients depend smoothly on $(t,x)$, and
\[
\Delta_{t_0}=\Delta.
\]

Set
\[
X
:=
\left\{
v\in C^{2,\alpha}(\overline{\Omega_0}):
v=0\ \hbox{on }\partial\Omega_0
\right\},
\qquad
Y:=C^\alpha(\overline{\Omega_0}),
\]
and define
\[
\mathcal F_L(t,v)
:=
\Delta_t v-e^v.
\]
Then
\[
\mathcal F_L(t_0,u_0)=0.
\]
The linearization with respect to $v$ at $(t_0,u_0)$ is
\[
D_v\mathcal F_L(t_0,u_0)[\varphi]
=
\Delta\varphi-e^{u_0}\varphi.
\]
Denote this operator by
\[
L_L\varphi
:=
\Delta\varphi-e^{u_0}\varphi.
\]
It is uniformly elliptic, and its zeroth-order coefficient satisfies
\[
-e^{u_0}<0.
\]
The maximum principle shows that
\[
L_L\varphi=0
\quad\hbox{in }\Omega_0,
\qquad
\varphi=0
\quad\hbox{on }\partial\Omega_0
\]
has only the trivial solution. Standard Schauder theory and the
Fredholm alternative therefore imply that
\[
L_L:X\longrightarrow Y
\]
is an isomorphism.

The Banach-space implicit function theorem gives a unique smooth
local branch
\[
t\longmapsto \widetilde v_t\in X
\]
such that
\[
\widetilde v_{t_0}=u_0,
\qquad
\mathcal F_L(t,\widetilde v_t)=0.
\]
Define
\[
\widetilde u_t
:=
\widetilde v_t\circ\Psi_t.
\]
Then
\[
\Delta\widetilde u_t=e^{\widetilde u_t}
\quad\hbox{in }\Omega_t,
\qquad
\widetilde u_t=0
\quad\hbox{on }\partial\Omega_t.
\]
By uniqueness,
\[
\widetilde u_t=u_t.
\]
Consequently,
\[
u_t\circ\Phi_t
=
\widetilde v_t
\longrightarrow
u_0
\quad\hbox{in }
C^{2,\alpha}(\overline{\Omega_0}).
\]

Finally, let
\[
g(q):=-\operatorname{arcosh}(e^{-q/2}),
\qquad q<0.
\]
For every $K\Subset\Omega_0$, one has
\[
u_0\leq-\eta_K<0
\quad\hbox{on }K
\]
for some $\eta_K>0$. Hence, for $t$ sufficiently close to $t_0$,
\[
u_t\circ\Phi_t\leq-\frac{\eta_K}{2}
\quad\hbox{on }K.
\]
Since $g$ is smooth on compact subintervals of $(-\infty,0)$,
continuity of composition in H\"older spaces gives
\[
w_t\circ\Phi_t
=
g(u_t\circ\Phi_t)
\longrightarrow
g(u_0)
=
w_{t_0}
\quad\hbox{in }C^{2,\alpha}(K).
\]
\end{proof}
We next prove the corresponding input for the Hessian equations. 

\begin{proposition}
\label{prop:local-stability-hessian}
Let $\alpha\in(0,1)$, and let
$\{\Omega_t\}_{t\in[0,1]}$ be a smooth family of bounded smooth
uniformly strictly convex domains with uniformly controlled geometry.

In the real case, assume that $\Omega_t\subset\mathbb R^n$, $n\geq 3$,
and that for every $t\in[0,1]$ the Dirichlet problem
\[
\begin{cases}
\sigma_2(D^2u_t)=e^{2u_t} & \text{in }\Omega_t,\\
u_t=0 & \text{on }\partial\Omega_t
\end{cases}
\]
admits a unique smooth $\sigma_2$-admissible solution
$u_t\in C^{2,\alpha}(\overline{\Omega_t})$.

In the complex case, assume that $\Omega_t\subset\mathbb C^m$,
$m\geq 2$, and that for every $t\in[0,1]$ the Dirichlet problem
\[
\begin{cases}
\sigma_2((u_t)_{i\bar j})=e^{2u_t} & \text{in }\Omega_t,\\
u_t=0 & \text{on }\partial\Omega_t
\end{cases}
\]
admits a unique smooth complex $\sigma_2$-admissible solution
$u_t\in C^{2,\alpha}(\overline{\Omega_t})$.

Fix $t_0\in[0,1]$. Then there exist smooth diffeomorphisms
\[
\Phi_t:
\overline{\Omega_{t_0}}
\longrightarrow
\overline{\Omega_t},
\qquad
\Phi_{t_0}=\operatorname{id},
\]
such that
\[
u_t\circ\Phi_t
\longrightarrow
u_{t_0}
\quad\text{in }
C^{2,\alpha}(\overline{\Omega_{t_0}})
\quad\text{as }t\longrightarrow t_0.
\]
Consequently, if
\[
w_t=-\operatorname{arcosh}(e^{-u_t/2}),
\]
then for every $K\Subset\Omega_{t_0}$,
\[
w_t\circ\Phi_t
\longrightarrow
w_{t_0}
\quad\text{in }C^{2,\alpha}(K)
\quad\text{as }t\longrightarrow t_0.
\]
\end{proposition}

\begin{proof}
Fix $t_0\in[0,1]$, and write
\[
\Omega_0:=\Omega_{t_0},
\qquad
u_0:=u_{t_0}.
\]
Since the domains form a smooth family, after restricting $t$ to a
sufficiently small neighborhood of $t_0$, there exists a smooth family
of diffeomorphisms
\[
\Phi_t:\overline{\Omega_0}\longrightarrow\overline{\Omega_t},
\qquad
\Phi_{t_0}=\operatorname{id}.
\]
Let
\[
\Psi_t:=\Phi_t^{-1}.
\]

For $v\in C^{2,\alpha}(\overline{\Omega_0})$, define the pulled-back
real Hessian by
\begin{equation}
 H_t[v](x)
:=
D_y^2(v\circ\Psi_t)(y)\big|_{y=\Phi_t(x)}.
\label{eq:pullback-hessian}
\end{equation}
In local coordinates,
\begin{equation}
\begin{aligned}
\bigl(H_t[v]\bigr)_{ij}
={}&
\left(
\frac{\partial\Psi_t^\mu}{\partial y_i}
\circ\Phi_t
\right)
\left(
\frac{\partial\Psi_t^\nu}{\partial y_j}
\circ\Phi_t
\right)
v_{\mu\nu}
\\
&+
\left(
\frac{\partial^2\Psi_t^\mu}
{\partial y_i\partial y_j}
\circ\Phi_t
\right)
v_\mu .
\end{aligned}
\label{eq:pullback-hessian-coordinate}
\end{equation}
Thus $ H_t[v]$ depends linearly on $D^2v$ and $Dv$, and its
coefficients depend smoothly on $(t,x)$. Moreover,
\begin{equation}
 H_{t_0}[v]=D^2v.
\label{eq:pullback-hessian-t0}
\end{equation}

Set
\begin{equation}
X
:=
\left\{
v\in C^{2,\alpha}(\overline{\Omega_0})
:
v=0\ \text{on }\partial\Omega_0
\right\},
\qquad
Y:=C^\alpha(\overline{\Omega_0}).
\label{eq:stability-banach-spaces}
\end{equation}

We first consider the real equation. Define
\begin{equation}
\mathcal F_{\mathrm R}(t,v)
:=
\sigma_2\bigl( H_t[v]\bigr)-e^{2v}.
\label{eq:real-pulledback-operator}
\end{equation}
This is a smooth map from an open neighborhood of $(t_0,u_0)$ in
$\mathbb R\times X$ into $Y$. Since $\Gamma_2$ is open and $u_0$ is
strictly admissible up to the boundary, the neighborhood can be chosen
so that
\[
H_t[v](x)\in\Gamma_2
\qquad
\text{for every }x\in\overline{\Omega_0}.
\]

Since $\Phi_{t_0}=\operatorname{id}$, one has
\[
\mathcal F_{\mathrm R}(t_0,u_0)=0.
\]
The linearization with respect to $v$ at $(t_0,u_0)$ is
\begin{equation}
D_v\mathcal F_{\mathrm R}(t_0,u_0)[\varphi]
=
T_1(D^2u_0)^{ij}\varphi_{ij}
-
2e^{2u_0}\varphi.
\label{eq:real-stability-linearization}
\end{equation}
Denote this operator by $L_{\mathrm R}$. Since
\[
D^2u_0(x)\in\Gamma_2
\qquad
\text{for every }x\in\overline{\Omega_0},
\]
one has
\[
T_1(D^2u_0(x))>0
\qquad
\text{on }\overline{\Omega_0}.
\]
By compactness of $\overline{\Omega_0}$, there exist constants
$\lambda,\Lambda>0$ such that
\begin{equation}
\lambda|\xi|^2
\leq
T_1(D^2u_0)^{ij}\xi_i\xi_j
\leq
\Lambda|\xi|^2
\label{eq:real-stability-uniform-ellipticity}
\end{equation}
for every $\xi\in\mathbb R^n$ and every
$x\in\overline{\Omega_0}$. Hence $L_{\mathrm R}$ is uniformly
elliptic.

Moreover, its zeroth-order coefficient satisfies
\[
-2e^{2u_0}<0.
\]
The maximum principle therefore implies that
\[
L_{\mathrm R}\varphi=0
\quad\text{in }\Omega_0,
\qquad
\varphi=0
\quad\text{on }\partial\Omega_0
\]
has only the trivial solution. Standard Schauder theory for uniformly
elliptic Dirichlet problems, together with the Fredholm alternative,
then shows that
\[
L_{\mathrm R}:X\longrightarrow Y
\]
is an isomorphism.

The Banach-space implicit function theorem consequently gives
$\varepsilon>0$ and a unique smooth map
\[
t\longmapsto\widetilde v_t\in X,
\qquad
|t-t_0|<\varepsilon,
\]
such that
\begin{equation}
\widetilde v_{t_0}=u_0,
\qquad
\mathcal F_{\mathrm R}(t,\widetilde v_t)=0.
\label{eq:real-local-branch}
\end{equation}
After decreasing $\varepsilon$ if necessary,
\[
H_t[\widetilde v_t]\in\Gamma_2.
\]
Define
\[
\widetilde u_t:=\widetilde v_t\circ\Psi_t.
\]
Then $\widetilde u_t$ is an admissible solution of
\[
\sigma_2(D^2\widetilde u_t)=e^{2\widetilde u_t}
\quad\text{in }\Omega_t,
\qquad
\widetilde u_t=0
\quad\text{on }\partial\Omega_t.
\]
By the uniqueness established in
Subsection~\ref{subsec:solvability-uniqueness}, we have
\[
\widetilde u_t=u_t.
\]
Therefore
\begin{equation}
u_t\circ\Phi_t
=
\widetilde v_t
\longrightarrow
u_0
\quad\text{in }
C^{2,\alpha}(\overline{\Omega_0})
\quad\text{as }t\longrightarrow t_0.
\label{eq:real-stability-convergence}
\end{equation}

We now consider the complex equation. Using the compression map $C$
defined in Subsection~2.3, define
\begin{equation}
\mathcal F_{\mathrm C}(t,v)
:=
\sigma_2\bigl(C( H_t[v])\bigr)-e^{2v}.
\label{eq:complex-pulledback-operator}
\end{equation}
Since
\[
C( H_{t_0}[v])
=
C(D^2v)
=
(v_{i\bar j}),
\]
the linearization at $(t_0,u_0)$ is
\begin{equation}
D_v\mathcal F_{\mathrm C}(t_0,u_0)[\varphi]
=
T_1((u_0)_{i\bar j})^{i\bar j}\varphi_{i\bar j}
-
2e^{2u_0}\varphi.
\label{eq:complex-stability-linearization}
\end{equation}
Denote this operator by $L_{\mathrm C}$.

Because $u_0$ is complex $\sigma_2$-admissible,
\[
T_1((u_0)_{i\bar j})>0
\qquad
\text{on }\overline{\Omega_0}.
\]
Thus $L_{\mathrm C}$ is uniformly elliptic when regarded as a real
second-order operator. More explicitly, if
$\xi\in\mathbb R^{2m}$ and $\zeta\in\mathbb C^m$ is the corresponding
complex vector, then
\begin{equation}
D_{D^2v}\mathcal F_{\mathrm C}(t_0,u_0)
[\xi\otimes\xi]
=
\frac14\,
\zeta^*
T_1((u_0)_{i\bar j})
\zeta
\geq
\lambda|\xi|^2
\label{eq:complex-stability-symbol}
\end{equation}
for some $\lambda>0$, uniformly on
$\overline{\Omega_0}$.

The zeroth-order coefficient is again strictly negative:
\[
-2e^{2u_0}<0.
\]
Hence the maximum principle, Schauder theory, and the Fredholm
alternative show that
\[
L_{\mathrm C}:X\longrightarrow Y
\]
is an isomorphism. The implicit function theorem therefore gives a
unique local smooth branch
\[
t\longmapsto\widetilde v_t\in X
\]
satisfying
\begin{equation}
\mathcal F_{\mathrm C}(t,\widetilde v_t)=0,
\qquad
\widetilde v_{t_0}=u_0.
\label{eq:complex-local-branch}
\end{equation}
For $t$ sufficiently close to $t_0$, this branch remains complex
$\sigma_2$-admissible. Define
\[
\widetilde u_t:=\widetilde v_t\circ\Psi_t.
\]
Then $\widetilde u_t$ is an admissible solution on $\Omega_t$ and,
By the uniqueness established in
Subsection~\ref{subsec:solvability-uniqueness},
\[
\widetilde u_t=u_t.
\]
Consequently,
\begin{equation}
u_t\circ\Phi_t
\longrightarrow
u_0
\quad\text{in }
C^{2,\alpha}(\overline{\Omega_0})
\quad\text{as }t\longrightarrow t_0.
\label{eq:complex-stability-convergence}
\end{equation}

It remains to treat the transformed functions. Let
\[
g(q):=-\operatorname{arcosh}(e^{-q/2}),
\qquad q<0.
\]
Fix $K\Subset\Omega_0$. Since $u_0<0$ in $\Omega_0$, there exists
$\eta_K>0$ such that
\[
u_0\leq-2\eta_K
\qquad\text{on }K.
\]
By either \eqref{eq:real-stability-convergence} or
\eqref{eq:complex-stability-convergence}, for $t$ sufficiently close
to $t_0$,
\[
u_t\circ\Phi_t\leq-\eta_K
\qquad\text{on }K.
\]
The function $g$ is smooth with bounded derivatives on every compact
subinterval of $(-\infty,0)$. Therefore the continuity of composition
in H\"older spaces gives
\[
w_t\circ\Phi_t
=
g(u_t\circ\Phi_t)
\longrightarrow
g(u_0)
=
w_{t_0}
\quad\text{in }C^{2,\alpha}(K).
\]
This proves the proposition.
\end{proof}

\medskip
\noindent\emph{Uniformity of the boundary collar.}
Propositions~\ref{prop:liouville-stability} and~5.2 imply that, after
pulling the solutions back to a fixed domain,
\[
u_t\circ\Phi_t
\longrightarrow
u_{t_0}
\quad\hbox{in }C^2(\overline{\Omega_{t_0}}).
\]
In particular, the boundary normal derivatives and the boundary
second derivatives vary continuously with $t$. Since
\[
\min_{\partial\Omega_{t_0}}
\partial_{\nu_{t_0}}u_{t_0}>0
\]
and the domains remain uniformly strictly convex, there exist
$c_0>0$ and $\delta>0$, independent of $t$ sufficiently close to
$t_0$, such that
\[
\partial_{\nu_t}u_t\geq c_0
\quad\hbox{on }\partial\Omega_t
\]
and
\[
D^2w_t>0
\quad\hbox{whenever}\quad
0<\operatorname{dist}(x,\partial\Omega_t)<\delta.
\]
Thus the boundary collar in
Lemma~\ref{lem:boundary-strict-convexity} may be chosen uniformly
for $t$ near $t_0$.

\subsection{Proofs of the Main Theorems}

We begin with Theorem \ref{thm:laplace-main}. Choose a Euclidean ball $B_R$
and connect it to $\Omega$ by Minkowski addition,
\begin{equation}\label{eq:minkowski-laplace}
\Omega_t=(1-t)B_R+t\Omega,
\qquad 0\leq t\leq1.
\end{equation}
The support functions satisfy
\[
h_t=(1-t)R+th_\Omega,
\]
so smooth uniform strict convexity is preserved throughout the deformation.
Let $u_t$ solve \eqref{eq:laplace-main} in $\Omega_t$ and set
\[
w_t=-\arcosh(e^{-u_t/2}).
\]
Define
\[
I_L=\{t\in[0,1]:D^2w_t>0\ \hbox{in }\Omega_t\}.
\]
Lemma \ref{lem:laplace-ball} gives $0\in I_L$. If $t_0\in I_L$, boundary
strict convexity gives a uniform boundary strip on which strict convexity
persists for nearby domains, while Proposition \ref{prop:laplace-stability}
preserves strict positivity on the remaining compact interior. Thus $I_L$ is
open.

If $t_j\in I_L$ and $t_j\to t_0$, Proposition
\ref{prop:laplace-stability} gives $D^2w_{t_0}\geq0$ in $\Omega_{t_0}$.
Proposition \ref{prop:laplace-constant-rank} then shows that
$\rank D^2w_{t_0}$ is constant. Boundary strict convexity gives full rank in
a boundary strip, so the rank is full throughout $\Omega_{t_0}$. Hence
$t_0\in I_L$, and $I_L$ is closed. Therefore $I_L=[0,1]$, proving Theorem
\ref{thm:laplace-main}.

We next prove Theorem \ref{thm:real-main}. Choose a Euclidean ball $B_R$ and connect it to $\Omega$ by Minkowski addition:
\begin{equation}\label{eq:minkowski-real}
\Omega_t=(1-t)B_R+t\Omega,\qquad 0\leq t\leq1.
\end{equation}
In terms of support functions,
\[
h_t=(1-t)R+th_\Omega.
\]
Uniform strict convexity of $B_R$ and $\Omega$ implies that the curvature-radius matrices remain uniformly positive along the deformation. Thus the domains remain smooth and uniformly strictly convex.

Let $u_t$ solve \eqref{eq:real-main} in $\Omega_t$ and set
\[
w_t=-\arcosh(e^{-u_t/2}).
\]
Define
\begin{equation}\label{eq:I-real}
I_{\R}
=
\{t\in[0,1]:D^2w_t>0\hbox{ in }\Omega_t\}.
\end{equation}
By Lemma \ref{lem:real-ball}, $0\in I_{\R}$. We show openness. If $t_0\in I_{\R}$, then $D^2w_{t_0}>0$ in $\Omega_{t_0}$. Boundary strict convexity gives a boundary strip in which $D^2w_t>0$ for $t$ close to $t_0$. On the remaining compact subset, strict positivity of $D^2w_{t_0}$ and Proposition \ref{prop:local-stability-hessian} imply $D^2w_t>0$ for $t$ close to $t_0$. Hence $I_{\R}$ is open.

We show closedness. Let $t_j\in I_{\R}$ and $t_j\to t_0$. By Proposition \ref{prop:local-stability-hessian},
\[
D^2w_{t_0}\geq0
\qquad\hbox{in }\Omega_{t_0}.
\]
The transformed admissibility condition follows from
\[
D^2u_{t_0}
=
h'(w_{t_0})
\left(
D^2w_{t_0}
-\frac1{s(w_{t_0})}Dw_{t_0}\otimes Dw_{t_0}
\right)
\]
and $h'>0$. Therefore Proposition \ref{prop:real-constant-rank} implies that $\rank D^2w_{t_0}$ is constant. Boundary strict convexity gives $D^2w_{t_0}>0$ in a boundary strip, so the constant rank is full. Hence $D^2w_{t_0}>0$ everywhere and $t_0\in I_{\R}$. Thus $I_{\R}$ is closed. Since $[0,1]$ is connected, $I_{\R}=[0,1]$, proving Theorem \ref{thm:real-main}.

The complex proof follows the same deformation argument, with the complex ball lemma and the complex constant-rank theorem replacing their real counterparts. Connect a ball in $\C^m\simeq\R^{2m}$ to $\Omega$ by the real Minkowski deformation
\[
\Omega_t=(1-t)B_R+t\Omega.
\]
Real uniform strict convexity is preserved by the support-function interpolation. Let $u_t$ solve \eqref{eq:complex-main} and set
\[
w_t=-\arcosh(e^{-u_t/2}).
\]
Define
\[
I_{\C}
=
\{t\in[0,1]:D^2_{\R}w_t>0\hbox{ in }\Omega_t\}.
\]
Lemma \ref{lem:complex-ball} gives $0\in I_{\C}$. Openness follows from boundary strict convexity and Proposition \ref{prop:local-stability-hessian}. For closedness, if $t_j\in I_{\C}$ and $t_j\to t_0$, then
\[
D^2_{\R}w_{t_0}\geq0.
\]
The transformed admissibility condition follows from
\[
(u_{t_0})_{i\bar j}
=
h'(w_{t_0})
\left(
(w_{t_0})_{i\bar j}
-\frac1{s(w_{t_0})}(w_{t_0})_i(w_{t_0})_{\bar j}
\right).
\]
Proposition \ref{prop:complex-constant-rank} gives constant rank of $D^2_{\R}w_{t_0}$. Since boundary strict convexity gives full rank in a boundary strip, the rank is full everywhere. Thus $I_{\C}$ is closed and $I_{\C}=[0,1]$. This proves Theorem \ref{thm:complex-main}.

\begin{remark}
The local $C^2$ stability used above is not a consequence of the constant-rank theorem. It is an independent compactness input. The exponential equation is simpler than the eigenvalue equation in one respect: it has no scaling invariance, and the local quadratic barrier in Proposition \ref{prop:laplace-stability} rules out collapse of subsequential limits to the zero function. This is the point at which the exponential right-hand side is used essentially.
\end{remark}

\small
\bibliographystyle{alpha}
\bibliography{reference}

@article{AlvarezLasryLions1997,
  author  = {Alvarez, O. and Lasry, J.-M. and Lions, P.-L.},
  title   = {Convex viscosity solutions and state constraints},
  journal = {J. Math. Pures Appl.},
  volume  = {76},
  year    = {1997},
  pages   = {265--288}
}

@incollection{BadianeZeriahi2023,
  author    = {Badiane, Papa and Zeriahi, Ahmed},
  title     = {A variational approach to the eigenvalue problem for complex {Hessian} operators},
  booktitle = {Nonlinear Analysis, Geometry and Applications},
  editor    = {Seck, Diaraf and Kangni, Kinvi and Sambou, Marie Salomon and Nang, Philibert and Fall, Mouhamed Moustapha},
  series    = {Trends in Mathematics},
  publisher = {Birkh{\"a}user},
  address   = {Cham},
  year      = {2024},
  pages     = {227--256},
  doi       = {10.1007/978-3-031-52681-7_10}
}

@article{BianGuan2009,
  author  = {Bian, B. and Guan, P.},
  title   = {A microscopic convexity principle for nonlinear partial differential equations},
  journal = {Invent. Math.},
  volume  = {177},
  year    = {2009},
  pages   = {307--335}
}

@article{BianGuan2010,
  author  = {Bian, B. and Guan, P.},
  title   = {A structural condition for microscopic convexity principle},
  journal = {Discrete Contin. Dyn. Syst.},
  volume  = {28},
  number  = {2},
  year    = {2010},
  pages   = {789--807}
}

@article{BauschkeGulerLewisSendov2001,
  author  = {Bauschke, H. H. and Guler, O. and Lewis, A. S. and Sendov, H. S.},
  title   = {Hyperbolic polynomials and convex analysis},
  journal = {Canad. J. Math.},
  volume  = {53},
  year    = {2001},
  pages   = {470--488}
}

@article{BrascampLieb1976,
  author  = {Brascamp, H. J. and Lieb, E. H.},
  title   = {On extensions of the {Brunn--Minkowski} and {Pr{\'e}kopa--Leindler} theorems},
  journal = {J. Funct. Anal.},
  volume  = {22},
  year    = {1976},
  pages   = {366--389}
}

@article{CaffarelliFriedman1985,
  author  = {Caffarelli, L. and Friedman, A.},
  title   = {Convexity of solutions of some semilinear elliptic equations},
  journal = {Duke Math. J.},
  volume  = {52},
  year    = {1985},
  pages   = {431--455}
}

@article{CaffarelliNirenbergSpruck1985,
  author  = {Caffarelli, L. and Nirenberg, L. and Spruck, J.},
  title   = {The {Dirichlet} problem for nonlinear second-order elliptic equations {III}: Functions of the eigenvalues of the {Hessian}},
  journal = {Acta Math.},
  volume  = {155},
  year    = {1985},
  pages   = {261--301}
}

@article{ChuLiuMcCleerey2024,
  author  = {Chu, Jianchun and Liu, Yaxiong and McCleerey, Nicholas},
  title   = {The eigenvalue problem for the complex {Hessian} operator on {$m$}-pseudoconvex manifolds},
  journal = {J. Funct. Anal.},
  volume  = {290},
  number  = {3},
  year    = {2026},
  pages   = {111258},
  doi     = {10.1016/j.jfa.2025.111258}
}

@article{Colesanti2005,
  author  = {Colesanti, A.},
  title   = {Brunn--Minkowski inequalities for variational functionals and related problems},
  journal = {Adv. Math.},
  volume  = {194},
  year    = {2005},
  pages   = {105--140}
}

@article{CollinsPicard2022,
  author  = {Collins, T. C. and Picard, S.},
  title   = {The {Dirichlet} problem for the {$k$}-{Hessian} equation on a complex manifold},
  journal = {Amer. J. Math.},
  volume  = {144},
  year    = {2022},
  pages   = {1641--1680}
}

@article{DinewKolodziej2011,
  author  = {Dinew, S{\l}awomir and Ko{\l}odziej, S{\l}awomir},
  title   = {A priori estimates for complex {Hessian} equations},
  journal = {Anal. PDE},
  volume  = {7},
  number  = {1},
  year    = {2014},
  pages   = {227--244},
  doi     = {10.2140/apde.2014.7.227}
}

@article{DinewKolodziej2012,
  author  = {Dinew, S{\l}awomir and Ko{\l}odziej, S{\l}awomir},
  title   = {Liouville and {Calabi--Yau} type theorems for complex {Hessian} equations},
  journal = {Amer. J. Math.},
  volume  = {139},
  number  = {2},
  year    = {2017},
  pages   = {403--415},
  doi     = {10.1353/ajm.2017.0009}
}

@article{Garding1959,
  author  = {G{\aa}rding, L.},
  title   = {An inequality for hyperbolic polynomials},
  journal = {J. Math. Mech.},
  volume  = {8},
  year    = {1959},
  pages   = {957--965}
}

@article{GuanMa2003,
  author  = {Guan, P. and Ma, X.-N.},
  title   = {The {Christoffel--Minkowski} problem {I}: Convexity of solutions of a {Hessian} equation},
  journal = {Invent. Math.},
  volume  = {151},
  year    = {2003},
  pages   = {553--577}
}

@article{HouMaWu2010,
  author  = {Hou, Z. and Ma, X.-N. and Wu, D.},
  title   = {A second order estimate for complex {Hessian} equations on a compact {K{\"a}hler} manifold},
  journal = {Math. Res. Lett.},
  volume  = {17},
  year    = {2010},
  pages   = {547--561}
}

@article{Jerison1996,
  author  = {Jerison, D.},
  title   = {The direct method in the calculus of variations for convex bodies},
  journal = {Adv. Math.},
  volume  = {122},
  year    = {1996},
  pages   = {262--279}
}

@article{Korevaar1983,
  author  = {Korevaar, N.},
  title   = {Convex solutions to nonlinear elliptic and parabolic boundary value problems},
  journal = {Indiana Univ. Math. J.},
  volume  = {32},
  year    = {1983},
  pages   = {603--614}
}

@article{KorevaarLewis1987,
  author  = {Korevaar, N. and Lewis, J.},
  title   = {Convex solutions of certain elliptic equations have constant rank {Hessians}},
  journal = {Arch. Rational Mech. Anal.},
  volume  = {97},
  year    = {1987},
  pages   = {19--32}
}

@article{LiuMaXu2010,
  author  = {Liu, P. and Ma, X.-N. and Xu, L.},
  title   = {A {Brunn--Minkowski} inequality for the {Hessian} eigenvalue in three-dimensional convex domain},
  journal = {Adv. Math.},
  volume  = {225},
  year    = {2010},
  pages   = {1616--1633}
}

@article{MaXu2008,
  author  = {Ma, X.-N. and Xu, L.},
  title   = {The convexity of solution of a class {Hessian} equation in bounded convex domain in {$\mathbb R^3$}},
  journal = {J. Funct. Anal.},
  volume  = {255},
  year    = {2008},
  pages   = {1713--1723}
}

@article{Renegar2006,
  author  = {Renegar, J.},
  title   = {Hyperbolic programs, and their derivative relaxations},
  journal = {Found. Comput. Math.},
  volume  = {6},
  year    = {2006},
  pages   = {59--79}
}

@article{Salani2005,
  author  = {Salani, P.},
  title   = {A {Brunn--Minkowski} inequality for the {Monge--Amp{\`e}re} eigenvalue},
  journal = {Adv. Math.},
  volume  = {194},
  year    = {2005},
  pages   = {67--86}
}

@article{Salani2012,
  author  = {Salani, P.},
  title   = {Convexity of solutions and {Brunn--Minkowski} inequalities for {Hessian} equations in {$\mathbb R^3$}},
  journal = {Adv. Math.},
  volume  = {229},
  year    = {2012},
  pages   = {1924--1948}
}

@book{Schneider1993,
  author    = {Schneider, R.},
  title     = {Convex Bodies: The {Brunn--Minkowski} Theory},
  publisher = {Cambridge University Press},
  address   = {Cambridge},
  year      = {1993}
}

@article{Wang1994,
  author  = {Wang, X.-J.},
  title   = {A class of fully nonlinear elliptic equations and related functionals},
  journal = {Indiana Univ. Math. J.},
  volume  = {43},
  year    = {1994},
  pages   = {25--54}
}

@article{limasa,
  author  = {Li, Jiahuan and Ma, Xi-Nan and Salani, Paolo},
  title   = {A {Brunn--Minkowski} inequality for the {Hessian} eigenvalue in convex domain},
  journal = {arXiv preprint},
  year    = {2026},
  note    = {arXiv:2606.22847}
}

@article{chenlima,
  author  = {Chen, Chuanqiang and Li, Jiahuan and Ma, Xi-Nan},
  title   = {{Brunn--Minkowski} inequality for the first complex {$\sigma_2$}-{Hessian} eigenvalue},
  journal = {arXiv preprint},
  year    = {2026},
  note    = {arXiv:2606.25678}
}

@article{ML71,
  author  = {Makar-Limanov, L. G.},
  title   = {Solution of the Dirichlet problem for the equation
             {$\Delta u=-1$} in a convex region},
  journal = {Math. Notes Acad. Sci. USSR},
  volume  = {9},
  year    = {1971},
  pages   = {52--53},
  doi     = {10.1007/BF01405053}
}

@article{Ken85,
  author  = {Kennington, Alan U.},
  title   = {Power concavity and boundary value problems},
  journal = {Indiana Univ. Math. J.},
  volume  = {34},
  number  = {3},
  year    = {1985},
  pages   = {687--704},
  doi     = {10.1512/iumj.1985.34.34036}
}

@article{CaffarelliGuanMa2007,
  author  = {Caffarelli, Luis A. and Guan, Pengfei and Ma, Xi-Nan},
  title   = {A constant rank theorem for solutions of fully nonlinear
             elliptic equations},
  journal = {Comm. Pure Appl. Math.},
  volume  = {60},
  number  = {12},
  year    = {2007},
  pages   = {1769--1791},
  doi     = {10.1002/cpa.20197}
}

\medskip
\noindent
(Jiahuan Li) Department of Mathematics, University of Science and Technology of China, Hefei, 230026, Anhui Province, China.\;
Email address: jiahuan@mail.ustc.edu.cn

\medskip
\noindent
(Shuning Xu) Department of Mathematics, University of Science and Technology of China, Hefei, 230026, Anhui Province, China.\;
Email address: xushuning@mail.ustc.edu.cn
\end{document}